\documentclass[11pt,reqno]{amsart}
\usepackage[margin=2cm]{geometry}
\usepackage{amsmath,amssymb,amsthm,mathtools}
\usepackage{enumitem}
\usepackage{microtype}
\usepackage{xcolor}
\usepackage{hyperref}
\hypersetup{colorlinks=true, linkcolor=blue, citecolor=purple, filecolor=magenta, urlcolor=blue}

\newcommand{\F}{\mathbb{F}}
\newcommand{\FF}{\mathbb{F}}

\newcommand{\1}{\mathbf{1}}

\newcommand{\NNN}{\mathcal{N}}
\newcommand{\HHH}{\mathcal{H}}
\newcommand{\eps}{\varepsilon}
\newcommand{\brac}[1]{\left\{#1\right\}}
\newcommand{\brap}[1]{\left(#1\right)}
\newcommand{\abs}[1]{\left|#1\right|}

\newcommand{\defeq}{\coloneqq}

\DeclareMathOperator*{\mindeg}{min\,deg}
\DeclareMathOperator*{\avgdeg}{avg\,deg}

\newtheorem{theorem}{Theorem}[section]
\newtheorem{lemma}[theorem]{Lemma}
\newtheorem{conjecture}[theorem]{Conjecture}
\newtheorem{proposition}[theorem]{Proposition}

\theoremstyle{definition}
\newtheorem{definition}[theorem]{Definition}

\DeclareSymbolFont{bbold}{U}{bbold}{m}{n}
\DeclareSymbolFontAlphabet{\mathbbold}{bbold}
\newcommand{\One}{\mathbbold{1}}

\newcommand{\pal}{\mathrm{pal}}

\definecolor{AfonsoBlue}{RGB}{30,65,123}

\definecolor{Lucas}{RGB}{150, 0, 0}

\definecolor{PNNOrange}{RGB}{242,140,40}

\title{The hypergraph Moore bound}
\date{\today}
\author{Afonso S. Bandeira}
  \address{ASB: Department of Mathematics, ETH Z\"urich, Rämistrasse 101, 8092 Z\"urich, Switzerland.}
  \email{bandeira@math.ethz.ch}  
  \author{Dmitriy Kunisky}
  \address{DK: Department of Applied Mathematics \& Statistics, Johns Hopkins University, 6225 Smith Avenue, Baltimore, MD 21209, USA}
  \email{kunisky@jhu.edu} 
  \author{Petar Nizi\'{c}-Nikolac}
  \address{PNN: Department of Mathematics, ETH Z\"urich, Rämistrasse 101, 8092 Z\"urich, Switzerland.}
  \email{petar.nizic-nikolac@ifor.math.ethz.ch} 
  \author{Lucas Pesenti}
  \address{LP: Department of Computer Science and Department of Mathematics, ETH Z\"urich, Rämistrasse 101, 8092 Z\"urich, Switzerland.}
  \email{lpesenti@ethz.ch} 
  \author{Robert Wang}
  \address{RW: Cheriton School of Computer Science, University of Waterloo, Canada. (Part of this work was done while RW was an academic guest in the Department of Mathematics at ETH Z\"urich)}
  \email{robert.wang2@uwaterloo.ca}

\begin{document}
\pagestyle{plain}
\begin{abstract}
    The hypergraph Moore bound conjectured by Feige (2008) controls the size of the smallest even cover in a $k$-uniform hypergraph in terms of the average density of hyperedges.
    An even cover is a set of hyperedges covering each vertex an even number of times, generalizing the notion of a cycle in a graph, so the size of the smallest non-trivial even cover provides a notion of hypergraph girth.
    Recent work, starting from the breakthrough result of Guruswami, Kothari, and Manohar (2022) proved the conjecture up to polylogarithmic factors, whose exponents were later gradually improved.
    We give a simple proof of Feige's original hypergraph Moore bound conjecture for all 
    $k \geq 3$, with no superfluous polylogarithmic factors.
    For the case of $k$ even, our proof roughly follows the proof of the graph Moore bound, but works with colored walks in a Kikuchi graph built from a hypergraph and controls their growth using the polynomial method. The argument is then extended to the case of $k$ odd by adapting a procedure in  [GKM22]. 
\end{abstract}

\maketitle

\section{Introduction}

The classical \emph{Moore bound} describes how large the girth of a graph can be given its average degree.
In particular, Alon, Hoory, and Linial~\cite{graphMooreIrregular} proved that every
graph on $n$ vertices with average degree $d> 2$ contains a cycle of length at most $2\lceil \log_{d-1} n\rceil+2$. 
In this work, we study the analogous question for hypergraphs.

A natural analog of girth for a hypergraph
is the size of its
smallest \emph{even cover}, where an 
even cover is a nonempty set of hyperedges
such that every vertex belongs to an even
number of hyperedges in that set. 
Feige~\cite{feigeConjecture} analyzed the size of the smallest even cover of random hypergraphs.
Based on the intuition that random hypergraphs
should be approximately extremal, he proposed
the following generalization of the Moore bound.

\begin{conjecture}[Feige's hypergraph Moore bound conjecture]\label{conj:feige}
    For every $k\ge 3$, there exist constants $c_k,C_k>0$ such that
    the following holds.
    For all integers $n\ge k$ and $1\le \rho \le n$, every
    $k$-uniform hypergraph on $n$
    vertices with
    \begin{equation}
        m\ge c_k n\left(\frac n\rho\right)^{\frac k2-1}\label{eq:feigeCondition}
    \end{equation}
    hyperedges contains an even
    cover of size at most $C_k \rho \log n$.
\end{conjecture}

There is an equivalent linear algebraic
formulation of this conjecture.
Represent each hyperedge $E\subseteq [n]$ by
its indicator vector $\mathbf 1_E\in \F_2^n$. Then,
a set of hyperedges $\{E_1, \ldots, E_q\}$ is an even cover precisely when
\begin{equation}
    \mathbf 1_{E_1} + \cdots + \mathbf 1_{E_q}=0\text{ in $\F_2^n$}\,. \label{eq:hyperedge-comb}
\end{equation}
Thus, Conjecture~\ref{conj:feige}
postulates that, when~\eqref{eq:feigeCondition} holds, any system of 
$k$-sparse linear
equations over $\mathbb F_2$ with $n$ variables and $m$ equations has a small
non-trivial linear dependency.
In coding theory, this translates into 
a bound on the rate-distance
tradeoff of low-density parity-check (LDPC) error-correcting codes.

In this paper, we prove Conjecture~\ref{conj:feige}.

\begin{theorem}[Hypergraph Moore bound]\label{thm:moore}
    Conjecture \ref{conj:feige} holds.
    In particular, when $k$ is even, one can take $c_k = 64$ and $C_k = 4k$.
\end{theorem}

We have made no attempt to optimize the constants. The significance of Theorem~\ref{thm:moore} is that the density requirement is at the conjectured scale~\eqref{eq:feigeCondition}, 
with no additional factor of $\log n$.
The case where $k$ is even contains most of the new ideas. The argument for the even case can then be combined with the techniques
in~\cite{HsiehKothariMohanty2023} (with appropriate modifications, as described in Section~\ref{sec:odd-uniformcase}) to
establish the odd case as well.

\subsection{Related Work} 
For even $k$, Naor and Verstraëte~\cite{naorVerstraete}
proved Feige's conjecture in the high-density
regime $\rho = O(1)$ (meaning, for $\rho$ constant independent of $n$). 
This remained the only progress
until the breakthrough work of Guruswami, Kothari, and
Manohar~\cite{GuruswamiKothariManohar2022}, who
obtained bounds throughout the full range of $\rho$, but with an additional
factor of $\log^{4k+1} n$ in the density requirement~\eqref{eq:feigeCondition}.
This was improved to a single extra $\log n$
factor independently by Munhá Correia and Sudakov~\cite{munhaPersonal} and by Hsieh, Kothari, and Mohanty~\cite{HsiehKothariMohanty2023}.
For $k=3$,~\cite{naorVerstraete} showed that $m\ge C n^{1.5}\log n$ suffices
to guarantee an even cover of size $O(\log n)$, and~\cite{feigeConjecture} subsequently improved the density requirement to $m\ge C n^{1.5}\sqrt{\log \log n}$. For general odd $k$,~\cite{HsiehKothariMohantyMunhaCorreiaSudakov2024} proved that Feige's conjecture
holds, up to an additional
factor of $\log^{1/(k+1)} n$ in~\eqref{eq:feigeCondition} 
for $\rho\le n/\log^2 n$. 
Theorem~\ref{thm:moore} removes this 
logarithmic factor for all $k\geq 3$ and thereby confirms Feige's
original conjecture.

Among these prior works, the approaches in~\cite{munhaPersonal,HsiehKothariMohantyMunhaCorreiaSudakov2024}, which use purely combinatorial arguments based on rainbow paths, are closest in spirit to ours.

Several important problems in combinatorics and computer science admit natural formulations in terms of even covers in hypergraphs.
Conjecture~\ref{conj:feige}
was originally motivated by the \emph{existence}
of polynomial-time verifiable certificates 
of unsatisfiability for
random constraint satisfaction problems
at densities strictly below those at which
known polynomial-time algorithms succeed in
\emph{finding} them~\cite{FKO-2006-WitnessRandom3SAT}. As a consequence of their
proof of Feige's conjecture up to polylogarithmic factors,~\cite{GuruswamiKothariManohar2022}
showed the existence of such certificates
in \emph{semirandom} instances of constraint
satisfaction problems.

To prove this result,~\cite{GuruswamiKothariManohar2022}
developed and popularized the ``Kikuchi method'' introduced in~\cite{WeinElAlaouiMoore2019} in the context of algorithms for Tensor PCA, which has 
since found numerous
applications in average-case complexity~\cite{GuruswamiKothariManohar2022,KothariXu26,BandeiraCipolloniSchroederVanHandel2026} and
coding theory~\cite{AlrabiahGuruswamiKothariManohar2023,KothariManohar2024LinearLCC,KothariManohar2024SmoothLCC,HsiehKothariMohantyMunhaCorreiaSudakov2024,JanzerManohar2025,BasuHsiehKothariLin2025}. In particular, the strongest known results in most of these applications retain
extra logarithmic factors analogous to the one that until now remained in the hypergraph Moore bound. 
The Kikuchi graphs underlying this method are also the starting point of our proof.

\subsection*{Acknowledgments and Use of AI}

We would like to thank Benny Sudakov and Kevin Lucca for several insightful discussions on the topic of this paper, and Benny Sudakov in particular for valuable feedback on an earlier version of this manuscript.

We used modern AI tools in this work, primarily GPT-5.6~Sol, but also GPT-5.5 Pro, Claude Opus~4.8, and Claude Fable~5.
In fact, the core technical innovation in the proof was found by GPT-5.6~Sol. 
The authors were working on a program to attempt to establish versions of Conjecture~\ref{conj:feige} using random matrix theory (as in the line of work \cite{GuruswamiKothariManohar2022,HsiehKothariMohanty2023}) and the analysis of certain limiting infinite-dimensional operators, leading to problems concerning lower bounds on numbers of combinatorial walks arising as tracial moments of these operators.
A separate publication is forthcoming on other implications of this program.
The argument found by GPT-5.6~Sol 
to establish the even-uniform hypergraph Moore bound 
was first given in a rather different language from how we have written it in this paper, but the crux of the argument was the same, using a diagonal polynomial system to bound the number of edges in a suitable subgraph of the hypercube.

The paper was written entirely by the authors. AI was used to scan the manuscript for errors and inconsistencies in notation.
All errors are ours.

LP's work on this
project was supported by the Swiss National Science Foundation, grant no.\ 10004947. RW's work on this project was supported by the Queen Elizabeth II Graduate Scholarship in Science and Technology (QEII-GSST).

\subsection*{Notation}

Throughout, let $k\geq 3$, $H\subseteq {\binom {[n]} k}$ be the edge set of a simple $k$-uniform hypergraph, where $[n] \coloneqq \{1, \ldots, n\}$, and $m\coloneqq\abs{H}$.

We will often work over $\F_2^H$, Boolean vectors indexed by (the hyperedges of) $H$.
For $E\in H$, let $\eps_E$ be the corresponding standard basis vector (this is equivalently $\1_{\{E\}}$ in the notation~\eqref{eq:hyperedge-comb}, but we reserve that notation exclusively for vectors in $\F_2^n$ to avoid confusion).

We denote the natural logarithm by
$\log$.
For a graph $G$, we write $V(G)$ for its vertex set, $\mathcal{E}(G)$ for its edge set (reserving the symbol $E$ for hypergraph edges), and $\mathcal{E}(A, B)$ for subsets $A, B \subseteq V(G)$ for the set of edges with one endpoint in $A$ and the other in $B$ (in the latter notation, the graph $G$ involved will be clear from context).

\section{Proof of Theorem~\ref{thm:moore} in the Even-Uniform Case}\label{sec:proof_even}

In this section, we suppose $k \geq 4$ is even, that the number of hyperedges of $H$
satisfies
\begin{equation}
    m\ge 64 n \left(\frac n \rho\right)^{\frac k2 - 1}\,,\label{eq:density_assumption}
\end{equation}
and we show that $H$ has an even cover of size at most $4k\rho \log n$.

\subsection{Proof Sketch}\label{sec:proofOverview}
The \emph{level-$\ell$ Kikuchi graph}
$K=K_\ell(H)$ associated to $H$ is the graph on $V(K) = \binom {[n]} \ell$ in which $S,T\in V(K)$
are adjacent if and only if $S\mathbin{\Delta} T\in H$.
Our main technical lemma towards the proof of Theorem~\ref{thm:moore} is a bound on
the average degree of neighborhoods in the Kikuchi graph.

\begin{lemma}\label{lem:densityNeighborhood}
    Let $g$ be the size of the smallest even cover of $H$, and let $\mathcal S\subseteq V(K)$ be contained in a ball of radius $\frac g 2 - 1$
    with respect to the graph distance on $K$.
    Then, $\avgdeg(K[\mathcal S])\le 2\ell$.
\end{lemma}

Here, $\avgdeg$ denotes the average degree and
$K[\mathcal N]$ denotes the subgraph of $K$ induced by $\mathcal N$.

Lemma~\ref{lem:densityNeighborhood} implies
Theorem~\ref{thm:moore} using a standard
ball-growing argument
that we briefly sketch (see Section~\ref{sec:outline} for details).
First, under the density assumption~\eqref{eq:density_assumption}, the Kikuchi graph of level $\ell\coloneqq \rho$ has average 
degree $\gg \ell$ (Proposition~\ref{prop:kikuchi-avg}). By iteratively deleting low-degree
vertices, we may in fact assume that $K$ has \emph{minimum} degree $\gg \ell$ (Lemma~\ref{lem:denseSubgraph}).
Fix any vertex of $K$ and iteratively grow a ball $\mathcal B_r$ of radius $r=1, 2,\ldots$ around it.
The minimum degree condition
shows that $\mathcal B_r$ has $\gg \ell |\mathcal B_r|$ adjacent edges;
these edges are fully contained in
$\mathcal B_{r+1}$, so Lemma~\ref{lem:densityNeighborhood} gives $|\mathcal B_{r+1}|\gg |\mathcal B_r|$. Thus the balls grow exponentially for as long as Lemma~\ref{lem:densityNeighborhood} applies. Since $K$ has only $\binom n \ell$
vertices, this growth can continue for at most
$O(\ell \log n)$ steps, which therefore must also be an upper bound on the girth of $H$.

Our main contribution is the proof of
Lemma~\ref{lem:densityNeighborhood} itself,
for which we use the polynomial method in its combinatorial form.
Let $Q$ denote the usual Boolean hypercube graph over
$V(Q) = \mathbb F_2^m$. We establish the following
lemma in Section~\ref{sec:interpolation},
which strengthens the usual edge-isoperimetric inequality for subsets of
the hypercube satisfying an additional algebraic assumption.

\begin{lemma}\label{lem:interpolationLemma}
    Let $\mathcal A\subseteq V(Q)$
    be such that for every $x\in \mathcal A$, the function
    $y\mapsto  \One \{x = y\}$ coincides on $\mathcal A$ with some polynomial over $\mathbb F_2$ of degree at most $d$.
    Then, $\avgdeg(Q[\mathcal A])\le 2d$.
\end{lemma}

The proof of Lemma~\ref{lem:interpolationLemma}
is a dimension-counting argument. By assumption, the monomials of degree at most
$d$ span the space of functions $\mathcal A\to \mathbb F_2$, so we may choose a basis
of $|\mathcal A|$ such monomials.
Partition the edges of $Q[\mathcal A]$ according to the coordinate in which their endpoints differ. For each $i\in [m]$, the subspace of functions $\mathcal A\to \mathbb F_2$ that are constant across
every edge in direction $i$ has codimension exactly
the number of these edges, since they form a matching. Every basis monomial
not containing $x_i$ belongs to this subspace, so the number
of edges in direction $i$ is at most the number
of monomials containing $x_i$. Summing over $i$,
each of the $|\mathcal A|$ monomials 
is counted at most $d$ times, so $Q[\mathcal A]$
has at most $d|\mathcal A|$ edges.

Finally, Lemma~\ref{lem:interpolationLemma} implies Lemma~\ref{lem:densityNeighborhood} by realizing $K[\mathcal S]$ as a subgraph of the hypercube (see Section~\ref{sec:palette}).
Let $\mathcal S$ be a ball of radius $r<(g-1)/2$
centered at $S_0$, and map each $S\in \mathcal S$ to its \emph{palette} in $\mathbb F_2^m$, which records
the set of hyperedges of $H$ appearing an odd
number of times along a shortest path
from $S_0$ to $S$ (Definition~\ref{def:palette}). 
If $S,T\in \mathcal S$ are adjacent, then their palettes are adjacent in the hypercube. Indeed, the two
chosen paths together with the edge $\{S,T\}$
form a closed walk of length at most $2r+1<g$. 
Since there
are no even covers of this size by assumption, the set of edge labels appearing an odd number of times along the walk is empty, and the
palettes of $S$ and $T$ differ exactly
in the coordinate corresponding to $S\mathbin{\Delta} T$.
Finally, for each vertex $u$,
the map $T\mapsto \One\{u\in T\}$ is linear in the palette of $T$. 
Taking the product
over all $u\in S$ gives a degree-$\ell$ polynomial implementing the delta function at $S$ (see~\eqref{eq:polynomial_even_case}).

The rest of this section is devoted to the full proof of Theorem~\ref{thm:moore}. 

\subsection{Large $\rho$ Case}\label{sec:small_density_case}

Assume \eqref{eq:density_assumption} holds.
We first handle the easier case of $\rho>n/(4k)$.
Denote by the \emph{weight} of a Boolean vector the number of entries equal to 1.
Consider the vectors $\1_E$ over $E \in H$.
Any non-trivial linear dependency between these vectors is precisely an even cover of $H$ of that size, giving a solution to~\eqref{eq:hyperedge-comb}.
Using the crude bound $m \geq 64n > n+1$, we conclude that any $n+1$ of these vectors are linearly dependent. 
Hence, some non-empty collection of at most $n+1$ distinct hyperedges is an even cover.
As $n+1<4k\rho\log n$, the result holds in this case.

From now on, we may therefore assume that
\begin{equation}
\rho\leq \frac{n}{4k}. \label{eq:rho-asm}
\end{equation}

\subsection{Reduction to a Dense Subgraph}
Set
    \begin{equation}
        \ell \coloneqq \max\left\{\frac{k}{2},\lceil\rho\rceil\right\}. \label{eq:ell-choice}
    \end{equation}
    Since $\lceil\rho\rceil\leq2\rho\leq(k/2)\rho$ and $\rho\leq n/(4k)$, we have
    \begin{equation*}
        \rho\leq\ell\leq \frac{k\rho}{2} \leq\frac n8,
    \end{equation*}
and so also $\frac k2\leq \ell\leq n-\frac k2$, which makes the level-$\ell$ Kikuchi graph well-defined.
We view an edge $\{S, T\}$ in the Kikuchi graph $K$ as being colored by $S \mathbin\Delta T$, so that $K$ is edge-colored by the set $H$.
We denote the size of the vertex set of $K$ by
\[ N \coloneqq \binom n\ell = |V(K)|. \]

For every $E \in H$, the number of edges in $K$ colored by $E$ is the same, given by multiplying the number of ways to partition $E$ into two equal parts of size $k/2$ and the number of ways to choose a further subset of $[n]$ of size $\ell - k/2$ disjoint from $E$, divided by two.
Using this in a double-counting calculation along with basic asymptotic estimates, we find the following.
\begin{proposition}
    \label{prop:kikuchi-avg}
    Let $K = K_{\ell}(H)$ be as above for arbitrary $\frac{k}{2} \leq \ell \leq n - \frac{k}{2}$.
    Then,
    \begin{equation}\label{eq:kikuchi_average_degree}
    \avgdeg\brap{K}
    = \frac{m\,
    \binom{k}{k/2}
    \binom{n-k}{\ell-k/2}}{\binom n\ell}.
    \end{equation}
    Further, if also $k \leq \frac{n}{4}$ and $\ell \leq \frac{n}{8}$, then
    \begin{equation}\label{eq:kikuchi_average_degree_est}
    \avgdeg(K) \geq m\left(\frac{\ell}{n}\right)^{k/2}.
    \end{equation}
\end{proposition}
The proof is given in Appendix~\ref{app:kikuchi_density}.

Using this general calculation, we find that in our setting $K$ must have a non-empty ``core'' dense subgraph of high minimum degree.
\begin{lemma}[Dense subgraph of Kikuchi graph]\label{lem:denseSubgraph}
    In the above setting (i.e., with $\rho$ satisfying~\eqref{eq:rho-asm}, $\ell$ as in~\eqref{eq:ell-choice}, and the assumption~\eqref{eq:density_assumption} holding), there exists a nonempty induced subgraph $K^\star$ of $K$ of minimum degree at least $32\ell$.
\end{lemma}
\begin{proof}
    An elementary argument shows that every finite graph of average degree $d$ contains a nonempty induced subgraph of minimum degree at least $d/2$.\footnote{To show this, repeatedly delete vertices of degree less than $d/2$ from a graph $G = (V, \mathcal{E})$. If every vertex were deleted, fewer than $(d/2)|V| = |\mathcal{E}|$ edges would have been deleted, which is impossible.}
    So, it suffices to verify that, under these assumptions, $\avgdeg(K) \geq 64\ell$.
    Indeed, we have by Proposition~\ref{prop:kikuchi-avg} combined with our assumptions,
    \begin{equation*}
        \avgdeg\brap{K}
        \geq
        m\brap{\frac \ell n}^{k/2}
        \geq
        64n\brap{\frac n\rho}^{k/2-1}\brap{\frac \ell n}^{k/2}
        =64\ell\brap{\frac \ell\rho}^{k/2-1}
        \geq64\ell. \qedhere
    \end{equation*}
\end{proof}

\subsection{From Low Density to Small Even Covers}
\label{sec:outline}

In this section, we prove Theorem~\ref{thm:moore}
assuming Lemma~\ref{lem:densityNeighborhood}.

Recall the idea of the proof of the graph Moore bound (the case $k = 2$): an even cover in a graph is just a union of cycles, and the absence of short cycles means that a graph is locally tree-like, in which case its neighborhoods grow exponentially fast if the minimum degree is large.
The absence of a cycle of length $\Theta(\log n)$ then implies a contradiction since there are not enough vertices for these neighborhoods to continue growing exponentially without exhausting all vertices of the graph.

Our argument is similar, but working with the Kikuchi graph. As in the proof of the graph Moore bound, we pass to a subgraph $K^\star$ with high enough minimum degree (in our case, $K^\star$ is a subgraph of the Kikuchi graph).
We cannot be so direct as to use the graph Moore bound on $K^\star$, since the Kikuchi graph (and so $K^\star$ as well) usually has short cycles, but ones that correspond only to trivial even covers. For example, 
suppose that $k = \ell = 4$, and $E_1 = \{1, 2, 3, 4\}$ and $E_2 = \{5, 6, 7, 8\}$ both belong to $H$.
Then,
\( \{1, 2, 5, 6\}, \{3, 4, 5, 6\}, \{3, 4, 7, 8\}, \{1, 2, 7, 8\} \)
form a 4-cycle in $K_{\ell}(H)$, but since the consecutive symmetric differences are $E_1, E_2, E_1, E_2$, this cycle corresponds to a trivial even cover in $H$.

Instead, we need to work only with the ``interesting'' cycles in $K^\star$.
To this end, call a cycle \emph{unpaired} if some edge color in the cycle occurs an odd number of times, and define the \emph{unpaired girth} of $K^\star$
as the minimal length of an unpaired cycle.
By construction, there is an even cover in $H$ of size at most the unpaired girth of $K^\star$
(consisting of the set of edge colors occurring an odd number of times in a shortest unpaired cycle).

\begin{proof}[Proof of Theorem~\ref{thm:moore} from Lemma~\ref{lem:densityNeighborhood}] 
    We argue by contradiction, and assume that $H$ contains no even cover of size at most $4k\rho\log n$.
    Define
    \begin{equation*}
        q \coloneqq \lfloor 2k\rho \log n \rfloor, 
    \end{equation*}
    which also satisfies, since $\rho \geq 2\ell / k$, that
    \begin{equation*}
    q \geq \lfloor 4\ell \log n \rfloor \geq \lfloor 4 \log N \rfloor > \log_4 N > 1.
    \end{equation*}

    Fix an arbitrary root $S_0 \in V(K^\star)$.
    For $0 \leq j \leq q - 1$, let $\NNN_j$ be the ball of radius $j$ centered at $S_0$ in 
    the graph distance of $K^{\star}$.
    We then have
    \[ \{S_0\} = \NNN_0 \subseteq \NNN_1 \subseteq \cdots \subseteq \NNN_{q - 1}. \]
    Combining that $\mindeg(K^\star) \geq 32\ell$ and Lemma~\ref{lem:densityNeighborhood}, we find that, for each $0 \leq j \leq q - 2$,
    \begin{equation}\label{eq:exponentialgrowthneighborhoods32and16}
        32\ell \cdot |\NNN_j| \leq \sum_{S \in \NNN_j} \deg_{K^\star}(S) \leq 2|\mathcal{E}(K^\star[\NNN_{j+1}])| = 2|\mathcal{E}(K[\NNN_{j+1}])| \leq 2\ell |\NNN_{j+1}|.
    \end{equation}
    where the second inequality uses that $\NNN_{j+1} \supseteq \NNN_j$ and that all edges with one endpoint in $\NNN_j$ have their other endpoint in $\NNN_{j + 1}$. Every path in $K^\star$ is a path in $K$, so $\NNN_{j+1}$ is contained in a ball of $K$ of radius $j+1$, and so the last inequality follows from Lemma~\ref{lem:densityNeighborhood}.  
    Rearranging, we find
    \[ |\NNN_{j+1}| \geq 16 |\NNN_j|. \]

    We apply this repeatedly, starting with $|\NNN_0| = 1$, until we reach a lower bound on $|\NNN_{q-1}|$, which reads
    \[ |\NNN_{q-1}| \geq 16^{q-1} \geq 4^q > N. \]
    $\NNN_{q-1}$ is a subset of vertices of the Kikuchi graph, of which there are only $N$ in total, so this is a contradiction.
    We conclude that $H$ contains an even cover of size at most $4k\rho\log n$.
\end{proof}

\subsection{Density Bound via Polynomial Method}\label{sec:interpolation}
In this section, we prove Lemma~\ref{lem:interpolationLemma}, whose proof is akin to that of Frankl-Wilson-type theorems in~\cite{alon1991multilinear}.
A similar result appears in the work of Haussler, Littlestone, and Warmuth~\cite{HausslerLittlestoneWarmuth1994}, phrased in terms of the one-inclusion graph
and the Vapnik-Chervonenkis (VC) dimension.

We give the following slightly more general version of Lemma~\ref{lem:interpolationLemma} that will also be useful for the odd case.
Given a product set $\Omega \coloneqq \Omega_1\times \dots\times \Omega_h$, we denote the Hamming distance on $\Omega$ by 
\[
    d_H(p,p') \coloneqq\left|\{i\in [h]\colon p[i] \neq p'[i]\}\right|
\]
for all $p,p' \in \Omega$.

\begin{lemma}[Polynomial method density bound]
    \label{lem:polynomial_block}
    Let $\Delta\ge 1$, $1\le d\le h$ be integers, $G = (V, \mathcal{E})$ be a graph, $\Omega_1, \dots, \Omega_h$ be finite sets, and $\Omega \coloneqq \Omega_1\times\cdots \times \Omega_h$. 
    Suppose that the vertices of $G$ 
    are assigned distinct labels $p_u\in \Omega$ for $u\in V$, such that
    $d_H(p_u,p_v)=1$ for every edge $\{u,v\}\in \mathcal E$.

    For each $i\in [h]$, let $G_i$ be the 
    spanning subgraph of $G$ whose edge set is $\mathcal E(G_i)\coloneqq \{\{u,v\}\in \mathcal E:p_u[i]\neq p_v[i]\}$. Suppose that each $G_i$ has maximum degree at most $\Delta$.

    Let $F_{\leq d}$ be the $\FF_2$-linear span of all functions $\Omega \to \FF_2$ that depend on at most $d$ coordinates.
    Suppose that, for each $u \in V$, there exist $f_u \in F_{\leq d}$ such that $f_u(p_v) = \One\{u = v\}$ for every $v\in V$.
    
    Then, $\avgdeg(G) \leq 2\Delta d$.
\end{lemma}

Lemma~\ref{lem:interpolationLemma} follows easily as a special case of Lemma~\ref{lem:polynomial_block} for $h=m$,
$\Omega_1=\ldots=\Omega_h = \FF_2$, and $G = Q[\mathcal A]$ (in which case each $G_i$ forms a matching, so $\Delta=1$).

\begin{proof}
    By assumption, the restrictions of functions in $F_{\le d}$ to $A \coloneqq \{p_v : v\in V\}$
    span all functions $A \to \FF_2$.
    Therefore, we may choose a collection
    $\mathcal B$ of functions $\Omega\to \FF_2$ that depend
    on at most $d$ coordinates and whose restrictions to $A$ form a basis for the
    space of functions $A\to \FF_2$.
    In particular, $|\mathcal B| = |A| = |V|$. For each $f\in \mathcal B$, fix
    a set $I_f\subseteq [h]$ of coordinates
    on which $f$ depends, with $|I_f|\le d$.

    For each $i\in [h]$, consider the subspace
    $W_i$ of functions $V\to \FF_2$ that are
    constant on every connected component of $G_i$. Then, $\dim W_i$ is the
    number of connected components of $G_i$.
    Now, observe that if $f\in \mathcal B$
    is such that $i\notin I_f$, then $u\mapsto f(p_u)$ is constant across every edge of $G_i$ (and so is also constant on every connected component). Hence,
    \[
        \dim W_i \ge \left|\{f\in \mathcal B: i\notin  I_f\}\right|\,.
    \]
    Finally, a graph with $n$ vertices, $c$ connected components and maximum degree $\Delta$ has at most $\Delta(n-c)$ edges.
    Since $G_1, \ldots, G_h$ form a partition
    of $G$, we obtain
    \[
        |\mathcal E| = \sum_{i=1}^h |\mathcal E(G_i)|\le\Delta \sum_{i=1}^h (|V(G_i)| - \dim W_i)\le \Delta \sum_{i=1}^h \left|\{f \in \mathcal B : i\in I_f\}\right| \le \Delta |V|d\,,
    \]
    where the last inequality uses double counting. This concludes the proof.
\end{proof}

\subsection{Palette Embedding}\label{sec:palette}

In this section, we prove Lemma~\ref{lem:densityNeighborhood} from Lemma~\ref{lem:interpolationLemma}.
Let $\mathcal S\subseteq V(K)$ be contained
in a ball $\NNN_j$ of radius $j\le \frac g2-1$ centered at $S_0$.
We describe the hypercube labeling that we will use to apply Lemma~\ref{lem:interpolationLemma}, associating with each $S \in \mathcal S$ a Boolean vector.
We call this the \emph{color palette} of $S$, because it is associated with the sequence of colors (edges of $H$) encountered along a walk from $S_0$ to $S$.

\begin{definition}[Color palette]\label{def:palette}
Assign to each vertex $S\in\mathcal S$ its \emph{color palette}, an element $\pal(S) \in \FF_2^H$, formed by taking a walk of length at most $j$ from $S_0$ to $S$ and recording which colors $E\in H$ were visited an odd number of times. Because the girth of $H$
is greater than $2j + 1$, this definition does not depend on the walk taken. 
Explicitly, given any such walk $S_0, \ldots, S_t = S$, for some $t \leq j$, we set
\[ \pal(S) \coloneqq \sum_{i = 1}^{t} \eps_{S_i \mathbin \Delta S_{i -1}}  \in \mathbb{F}_2^H.\]
\end{definition}
\noindent
Note also that the color palette $\pal(S)$ together with the choice of root $S_0$ determines $S$, since from $\pal(S)$ we may compute $\1_{S_0} + \sum_{i = 1}^t (\1_{S_i} + \1_{S_{i-1}}) = \1_S$.
We also observe that, if two $S_1,S_2 \in \mathcal S$ are adjacent as vertices of $K$, then their palettes must have Hamming distance~1.
Hence, $K[\mathcal S]$ is isomorphic
to (a subgraph of) the subgraph of the hypercube induced by
$\{\pal(S) : S\in \mathcal S\}$.

It remains to produce a suitable collection of low-degree polynomials $f_S$ for $S \in \mathcal S$.
Consider the following polynomials:
\begin{equation}
    f_S(X) = \prod_{v \in S}\left(\One\{v \in S_0\} + \sum_{E \in H: v \in E} X_E\right).\label{eq:polynomial_even_case}
\end{equation}
To see why these polynomials satisfy the condition of Lemma~\ref{lem:interpolationLemma}, consider the evaluation with $X = \pal(T)$.
The expression $\One\{v \in S_0\} + \sum_{E \in H: v \in E} X_E$ counts the parity of the number of times vertex $v$ is visited along a walk from $S_0$ to $T$, which, by the definition of a walk in the Kikuchi graph, is simply $\One\{v \in T\}$.
Thus we have
\begin{equation*}
    f_S(\pal(T)) = \prod_{v \in S} \One\{v \in T\} = \One\{S \subseteq T\} = \One\{S = T\}, 
\end{equation*}
the final step following because $|S| = |T|$.
We emphasize that this simple last step is a crucial consequence of working over the Kikuchi graph defined on sets of a fixed cardinality (equivalently, a slice of the hypercube of fixed weight $\ell$), and that on this slice inclusion of sets implies equality. 

We also have $\deg(f_S) \leq |S| = \ell$.
Applying Lemma~\ref{lem:interpolationLemma} with this choice of polynomials $f_S$ implies
Lemma~\ref{lem:densityNeighborhood}.

\section{Proof of Theorem~\ref{thm:moore} in the Odd-Uniform Case
}\label{sec:odd-uniformcase}

In this section, we suppose $k \geq 3$ is odd.
We prove Theorem~\ref{thm:moore} by combining the proof from Section \ref{sec:proof_even} for even-uniform hypergraphs with the procedure in~\cite{GuruswamiKothariManohar2022,
HsiehKothariMohanty2023,
HsiehKothariMohantyMunhaCorreiaSudakov2024} to handle odd-uniform hypergraphs, with some appropriate adaptations.

Let $B_k$ be a sufficiently large constant, depending only on $k$.

\subsection{Large $\rho$ Case}
If $\rho>n/B_k$, increasing constants $c_k$ and $C_k$ in Theorem \ref{thm:moore} allow us to use the same argument from Section \ref{sec:small_density_case}.
In particular, choosing
\begin{equation*}
    c_k \geq B_k^{\frac k2-1},\qquad\qquad
    C_k \geq B_k + 1
\end{equation*}
guarantees that there is a collection of $n+1$ linearly dependent vectors that form an even cover of the desired length.
From now on, we may therefore assume that
\begin{equation*}
\rho\leq \frac{n}{B_k},
\end{equation*}
and follow~\cite{HsiehKothariMohanty2023} closely.\footnote{We chose not to reconstruct many of the arguments and descriptions in~\cite{HsiehKothariMohanty2023} and encourage the reader to consult~\cite{HsiehKothariMohanty2023}.} 

\subsection{Hypergraph Decomposition}

Set $\ell\coloneqq \max\{ 2k,\lceil{\rho\rceil} \}$, so that $\rho\leq \ell \leq 2k\rho$.
For $1\leq i\leq k-1$, set $\tau_i \defeq \max\{ 2, \left(\frac{n}{\ell} \right)^{\frac{k}2-i} \}$.\footnote{
Note that $\left(\frac{n}{\ell} \right)^{\frac{k}2-i}$ may not be an integer, we do not explicitly take a ceiling to not overload notation.}
We apply~\cite[Algorithm 1]{HsiehKothariMohanty2023}, which partitions the hyperedges $H$ into $k$ sets
\[
    H = \HHH^{(0)} \sqcup \HHH^{(1)} \sqcup \cdots \sqcup \HHH^{(k-1)} 
\]
having the following properties (see~\cite{HsiehKothariMohanty2023} for a proof):
\begin{enumerate}
    \item There exists $0<i<k$ (in particular $i\neq 0$) such that $|\HHH^{(i)}|\geq \frac m k$.\footnote{By the pigeonhole principle, there exists such $0\leq i <k$, but, as~\cite{HsiehKothariMohanty2023} shows, if $|\HHH^{(0)}|\geq \frac{m}{k}$, there would be a node $v \in [n]$ in $\frac{k|\HHH^{(0)}|}{n}\geq\frac{m}{n} \geq c_k \left(\frac n\rho\right)^{\frac k2-1}$ in $\HHH^{(0)}$ hyperedges, so some of these hyperedges would have been added to $\HHH^{(1)}$ in Algorithm 1 of~\cite{HsiehKothariMohanty2023}.}
    
    \item For all $1\le i<k$, we can partition $\HHH^{(i)} = H_1^{(i)} \sqcup \cdots \sqcup H_{h_i}^{(i)}$ such that, for all $j\in [h_i]$, $|H_j^{(i)}|\geq\tau_i$,
    and there exists $U^{(i)}_j\in \binom {[n]} i$ such that $U^{(i)}_j\subseteq E$ for all $E\in H_j^{(i)}$. For $i\geq \frac{k+1}{2}$, $|H_j^{(i)}|=\tau_i=2$. Moreover, $h_i\leq \frac{m}{\tau_i}$.
    \item For any $1\le i,s$ such that $i+s<k$ and $U\in \binom {[n]} {i+s}$, the number of hyperedges in $\HHH^{(i)}$ containing $U$ is less than $\tau_{i+s}$ (in particular, if $i\geq \frac{k+1}2$, it is at most 1).
\end{enumerate}

The subsets $H_j^{(i)}$ are called \emph{buckets}, and we refer to the subset $U_j^{(i)}$ shared by hyperedges in the bucket as their \emph{center}. 

\subsection{Large Centers: Even-Uniformity Bound} 

In this section, we treat the case where
condition (1) above (i.e., $|\HHH^{(i)}|\geq \frac{m}{k}$) is achieved for some $\frac {k+1} 2 \le i<k$. 
Conditions (2) and (3) above imply that for all $j\in [h_i]$, $H^{(i)}_j$ consists of a single pair of hyperedges $\{E_j,F_j\}$, whose intersection size is exactly $i$.
 
Now construct a $2(k-i)$ even-uniform hypergraph, $\hat{H}$, with edges $\{E_j\mathbin \Delta F_j:j\in [h_i]\}$. The proof of~\cite[Lemma~3.4]{HsiehKothariMohanty2023} shows that assuming $H$ has no even cover of size $4$, if $\hat{H}$ has an even cover of size $q$, then $H$ has an even cover of size at most $2q$. By our assumption, $\hat{H}$ has at least $\frac{m}{2k}$ hyperedges. Assuming $m\geq c_kn(n/\rho)^{k/2-1}$, the even case of Theorem ~\ref{thm:moore}, together with the ordinary Moore bound (e.g.~\cite[Proposition 2.1]{HsiehKothariMohanty2023}), implies $\hat{H}$ has an even cover of size at most $O_k(\rho\log{n})$.

\subsection{Small Centers: Colored Kikuchi Graph}

We focus in the remaining case: suppose that
there exists $i\leq \frac{k-1}{2}$ such that $|\HHH^{(i)}|\geq \frac{m}{k}$. We fix $i$
from now on and write $H_j\coloneqq H^{(i)}_j$ and $U_j \defeq U^{(i)}_j$.
For each $E\in H_j$, we denote by $\tilde{E}\defeq E\setminus U_j$ the hyperedge with its center removed. 

The argument proceeds by building the {\em colored Kikuchi graph $K_c$}~\cite[Definition 3.5]{HsiehKothariMohanty2023}. We duplicate each node $v\in[n]$ into  two labeled versions $(v,1)$ and $(v,2)$ of $v$ (called green and blue in~\cite{HsiehKothariMohanty2023}). 
Then, $K_c$ is a graph with vertex set
\[
    V(K_c) = {\binom {[n]\times \{1,2\}} \ell}\,.
\]
Given $S\in V(K_c)$, we denote by $S^{(1)} = \{v\in [n] : (v,1)\in S\}$ and similarly for $S^{(2)}$.

We now describe the edge set of $K_c$.
A pair $(E,F)\in H_j\times H_j$ with $E\neq F$ forms an edge in $K_c$ between $S,T\in V(K_c)$ if the following three conditions hold:
\begin{eqnarray*}
    S^{(1)}\mathbin{\Delta} T^{(1)} &=& \tilde{E} \\ 
    S^{(2)}\mathbin{\Delta} T^{(2)} &=& \tilde{F} \\
    \left( \left| \tilde{E} \cap S^{(1)} \right| , \left| \tilde{F} \cap S^{(2)} \right| \right) &\in& \left\{ \left(\left\lceil\frac {k-i} 2\right\rceil,\left\lfloor\frac {k-i} 2\right\rfloor\right),\left(\left\lfloor\frac {k-i} 2\right\rfloor,\left\lceil\frac {k-i} 2\right\rceil\right) \right\},
\end{eqnarray*}
in which case the edge is labeled $(E,F)$.

Note that, assuming the girth is larger than four, $K_c$ is a simple graph. Indeed, if there were $S,T\in V(K_c)$ connected by two distinct edges labels $(E,F)$ and $(E',F')$ then either $E=F'$ (and $F=E'$) or $E,F,E',F'$ forms an even cover of size $4$. If $E=F'$ then $E$ and $E'$ would have to belong to the same bucket, and $\tilde{E}=\tilde{E'}$, so $E=E'$, and so the edge labels were not distinct.

\subsection{Pruning Colored Kikuchi Graph}
The next step in \cite{HsiehKothariMohanty2023} is to prune $K_c$ so that each bucket of hyperedges contributes only a low-degree subgraph to the Kikuchi matrix. For $S\in V(K_c)$, let $d_E(S)$ be the number of edges in $K_c$ incident on $S$ whose label contains $E$.
Let $K_c^\star$ be the subgraph of $K_c$ consisting of edges $\{S,T\}$, with label $(E,F)$, such that
\begin{equation}\label{eq:condition_given_by_pruning}
    d_E(S)=d_F(S)=d_E(T)=d_F(T)=1.
\end{equation}
A key observation in~\cite{HsiehKothariMohanty2023} is that $K_c$ remains dense after pruning.

\begin{proposition}[{\cite[Claim 3.10]{HsiehKothariMohanty2023}}] 
    \label{prop:odd-fixed-role-package}
    There is a constant $\gamma_k>0$, depending only on $k$, such that, if
    $B_k$ is sufficiently large, then
    \begin{equation*}
        \avgdeg(K_c^\star)\geq
        \gamma_km\brap{\frac\ell n}^{k/2}.
    \end{equation*}
\end{proposition}

At this point we need deviate slightly from the argument in~\cite{HsiehKothariMohanty2023} and for every original hyperedge $E$, we retain only one of $E^{(1)}$ and $E^{(2)}$. We do this by building a map $\sigma:\HHH^{(i)}\to\brac{1,2}$ and retaining only an edge with label $(E,F)$ if $\sigma(E)=1$ and $\sigma(F)=2$. A routine probabilistic method argument can show such a map exists while keeping at least $\frac14$ of the edges. One can then show, that for each state $S\in V(K_c^\star)$, and each (post-prunned) bucket $H_j^\star$, there are only at most two edges incident on $S$ with labels in $H_j^\star$. This means that, each bucket $H_j^\star$, contributes a subgraph of $K_c^\star$ of degree at most 2. By another routine argument, one can take a subset of these edges, of size at least $1/3$, that forms a matching. The subgraph formed by these matchings has a number of edges at least $1/12$ of $H_j^\star$. A direct analogue of Lemma~\ref{prop:kikuchi-avg} provides a subgraph $K^{\star\star}$ with a large minimum degree, on which we will do the neighborhood expansion argument. This construction is done in the Lemma below.

\begin{lemma}[Dense core with bucket matchings]\label{lemma:densecorebucketmatchings}
    There is a subgraph $K_c^{\star\star}$ of $K_c^{\star}$ 
    such that:
    \begin{enumerate}
        \item There is a map $\sigma\colon \HHH^{(i)} \to \brac{1,2}$ such that the edge with the label $(E,F)$ stays if $\sigma(E) = 1$ and $\sigma(F) = 2$.
        
        \item For every state $S \in V(K_c^{\star})$ and every bucket $H_j^{\star}$ there is at most one edge whose label $(E,F)$ contains hyperdges $E,F$ from the bucket $H_j^{\star}$.
        
        \item Its minimum degree is lower bounded by
        \begin{equation*}
            \mindeg(K_c^{\star\star})\geq
            \frac{1}{24}\gamma_km\brap{\frac\ell n}^{k/2}.
        \end{equation*}
    \end{enumerate}
\end{lemma}

\begin{proof}
    If a map $\sigma$ is chosen uniformly at random, every edge in $K_c^{\star}$ is retained with probability $1/4$. Hence there is some map ensuring that the density is reduced by a factor $\leq 4$.
    
    Fix a state $S$ and suppose for the sake of contradiction that there are three edges incident with labels in $H_j^{\star}$. By pigeonhole principle, one can find two edges with labels $(E,F)$ and $(E',F')$ incident such that
    \begin{equation*}
        \left( \left| \tilde{E} \cap S^{(1)} \right| , \left| \tilde{F} \cap S^{(2)} \right| \right)
        =
        \left( \left| \tilde{E}' \cap S^{(1)} \right| , \left| \tilde{F}' \cap S^{(2)} \right| \right).
    \end{equation*}
    We claim that $E, F, E', F'$ are four distinct edges, i.e. all six pairs are distinct. Four of them are guaranteed since $\sigma(E) = \sigma(E') = 1$ and $\sigma(F) = \sigma(F') = 2$.
    For the remaining two, if either $E = E'$ or $F = F'$, by the pruning condition \eqref{eq:condition_given_by_pruning} both must hold (as otherwise $d_E(S) \geq 2$ or $d_F(S) \geq 2$), but then $(E,F) = (E',F')$, contradiction.

    As these are four distinct hyperedges, then $S$ is incident to an edge $(E,F')$ in $K_c$, which is different from $(E,F)$, again yielding contradiction by \eqref{eq:condition_given_by_pruning}.

    Thus the set of edges with labels from $H_j^\star$ is currently a vertex-disjoint union of a paths and cycles.
    One now reduces a density by a factor $\leq 3$ to guarantee that the edges with labels from $H_j^\star$ form a matching.
    
    Finally, one reduces a density by a factor $\leq 2$ to produce a dense core as argued in Lemma \ref{lem:denseSubgraph}.    
\end{proof}

\subsection{Using the Polynomial Method}
We proceed by utilizing the polynomial method density bound (Lemma \ref{lem:polynomial_block}) to show the following analogue of Lemma \ref{lem:densityNeighborhood}.

\begin{lemma}\label{lem:densityNeighborhoododd}
    Let $g$ be the size of the smallest even cover of $H$, and let $\mathcal S\subseteq V(K_c^{\star\star})$ be contained in a ball of radius $\frac g 4 - 1$
    with respect to the graph distance on $K_c^{\star\star}$.
    Then, $\avgdeg(K_c^{\star\star}[\mathcal S])\le 2\ell$.
\end{lemma}

As before, let $S_0 \in V(K_c^{\star\star})$ so that $\mathcal S$ is contained in the ball of radius $j \leq \frac g 4 - 1$ around $S_0$. We adapt the color palette (Definition \ref{def:palette}) as follows.

\begin{definition}[Color palette for colored Kikuchi]\label{def:adapted_palette}
Assign to each vertex $S\in\mathcal S$ its \emph{color palette}, an element $\pal(S) \in \FF_2^H$, formed by taking a walk of length $t \leq j$ from $S_0$ to $S$, using color-pairs $(E_{2i-1},E_{2i})$, and recording which hyperedges from $\brac{E_1,\ldots,E_{2t}}$ were visited an odd number of times. Because the girth of $H$
is greater than $4j + 1$, this definition does not depend on the walk taken. 
Explicitly, given any such walk, we set
\[ \pal(S) \coloneqq \sum_{i = 1}^{t} \brap{\eps_{E_{2i-1}} + \eps_{E_{2i}}}  \in \mathbb{F}_2^H.\]
\end{definition}
\noindent
As before, the color palette $\pal(S)$ together with the choice of root $S_0$ determines $S$ (recall that, 
for every hyperedge $E$ only $E^{(\sigma(E))}$ is present in edge color-pairs of $K_c^{\star\star}$).

Set $h = h_i$ to be the number of buckets.
While $K_c^{\star\star}[\mathcal S]$ is not isomorphic
to the subgraph of the hypercube, the palettes of any two adjacent vertices in $K_c^{\star\star}$ differ only in coordinates of one bucket. I.e. for any $S \sim_{K_c^{\star\star}} T$ there is unique $j \in [h]$ for which
\begin{equation*}
    \mathrm{supp}\brap{\pal(S) + \pal(T)} \subseteq H_j^{\star\star}.
\end{equation*}

Therefore neighboring palettes have Hamming distance one with respect to the product set
\begin{equation*}
    \Omega = \F_2^{H_1} \times \ldots \times \F_2^{H_h}.
\end{equation*}

Now we construct the polynomials akin to \eqref{eq:polynomial_even_case}, by setting
\begin{equation*}
    f_S(X) = \prod_{(v,\alpha) \in S}\left(\One\{(v,\alpha) \in S_0\} + \sum_{E \in \HHH^{(i)}:\, v \in \tilde E, \alpha = \sigma(E)} X_E\right).
\end{equation*}
Likewise $f_S(\pal(T)) = \One\{S = T\}$.

These polynomials are of degree $\ell$, thus in the span of the monomials (which are functions that depend only on $\ell$ coordinates). Lemma \ref{lem:polynomial_block} then yields $\avgdeg(K_c^{\star \star}[\mathcal S]) \leq 2\ell$.

Now that we have established Lemma~\ref{lem:densityNeighborhoododd}, the rest of the proof proceeds as in the even case. Due to Lemma~\ref{lemma:densecorebucketmatchings} we know the minimum degree in $K_c^{\star\star}$ is at least $\frac1{24}\gamma_k m \left(\frac{\ell}{n} \right)^{k/2}$. From \eqref{eq:feigeCondition} (and the fact that $\rho\leq \ell$), we have
\[
\mindeg\left(K_c^{\star\star}\right) \geq
\frac1{24}\gamma_k c_k n \left(\frac{n}{\rho} \right)^{k/2-1} \left(\frac{\ell}{n} \right)^{k/2} \geq \frac{\gamma_k c_k}{24}\ell.
\]
Pick $c_k$ large enough so that $\frac{\gamma_k c_k}{24}\geq 32$, so $\mindeg\left(K_c^{\star\star}\right)\geq 32\ell$. 
As in the proof of the even case, we fix a vertex, $S_0$, and let $\NNN_j$ be neighborhood of radius $j$ around $S_0$ in $K_c^{\star \star}$. Since, for $j\leq \frac{g}4-2$, Lemma~\ref{lem:densityNeighborhoododd} implies $|\mathcal{E}(K_c^{\star \star}[\NNN_{j+1}])|\leq \ell |\NNN_{j+1}|$, the inequality \eqref{eq:exponentialgrowthneighborhoods32and16} holds.

We set 
\begin{equation*}
    q \coloneqq \left\lfloor \frac14 C_k\rho\log n \right\rfloor,
\end{equation*}
and we finish as in the even case.

\appendix

\section{Proof of Proposition~\ref{prop:kikuchi-avg}}\label{app:kikuchi_density}

We first show~\eqref{eq:kikuchi_average_degree}.
For a given $E \in H$, the number of pairs $(S, T) \in V(K)^2$ such that $S \mathbin \Delta T = E$ is
\[
    \binom{k}{k/2}\binom{n - k}{\ell - k/2}.
\]
Thus,
\[
    \sum_{S \in V(K)} \deg(S) = m\binom{k}{k/2}\binom{n - k}{\ell - k/2},
\]
and the formula for $\avgdeg(K)$ follows upon dividing by $|V(K)| = \binom{n}{\ell}$.

We next show \eqref{eq:kikuchi_average_degree_est}.
Define $s \coloneqq k/2$.
Then, \eqref{eq:kikuchi_average_degree} can be rewritten as
\begin{equation}
\avgdeg(K) = \binom{k}{s} \cdot m \cdot \frac{(\ell)_s(n - \ell)_s}{(n)_{2s}}, \label{eq:kikuchi_average_degree_falling}
\end{equation}
where $(a)_b = a(a - 1) \cdots (a - b + 1)$ is the falling factorial.
Since $\ell \geq s$ by assumption, we have the estimate
\begin{equation*}
        (\ell)_s
        \geq
        \frac{s!}{s^s}\ell^s.
    \end{equation*}
    Moreover, $\ell\leq n/8$ and $s\leq n/8$, so
    \begin{equation*}
        (n-\ell)_s \geq \brap{\frac{3n}{4}}^s, \qquad\qquad
        (n)_{2s} \leq n^{2s}.
    \end{equation*}
    Combining \eqref{eq:kikuchi_average_degree_falling} with these estimates, we find that
    \begin{equation*}
        \avgdeg(K)
        \geq
        \beta_s m\brap{\frac \ell n}^s,
    \end{equation*}
    where
    \[ \beta_s
        =
        \binom{2s}{s}\frac{s!}{s^s}\brap{\frac34}^s. \]
    We claim that $\beta_s>1$. Indeed,
    \begin{equation*}
        \beta_s
        =
        \brap{\frac34}^s
        \prod_{j=1}^s\brap{1+\frac js}.
    \end{equation*}
    Since $x\mapsto\log(1+x)$ is increasing,
    \begin{equation*}
        \frac1s\sum_{j=1}^s\log\brap{1+\frac js}
        \geq
        \int_0^1\log(1+x)\,\mathrm dx
        =2\log 2-1.
    \end{equation*}
    Therefore
    \begin{equation*}
        \log\beta_s
        \geq
        s\brap{\log(3/4)+2\log 2-1}
        =s(\log 3-1)>0,
    \end{equation*}
    completing the proof.

\raggedbottom
\bibliographystyle{alpha}
\bibliography{main}

\end{document}